\documentclass[12pt,fleqn,leqno]{amsart}
\usepackage{amsfonts,amssymb}
\usepackage{mathrsfs}
\usepackage[bb=boondox]{mathalfa} 
\usepackage{enumitem}

\usepackage[svgnames]{xcolor}
\usepackage[colorlinks,linkcolor=blue,citecolor=Green]{hyperref}
\usepackage{titlesec}

\titleformat{\section}[block]
 {\bfseries}
 {\thesection.}
 {\fontdimen2\font}
 {}

\setlist{noitemsep}
\setenumerate{labelindent=\parindent,label=\upshape{(\alph*)}}

\newtheorem{theorem}{Theorem}[section]
\newtheorem{corollary}[theorem]{Corollary}
\newtheorem{proposition}[theorem]{Proposition}

\theoremstyle{definition}
\newtheorem{remark}[theorem]{Remark}

\DeclareMathOperator{\Q}{\mathbb{Q}}
\DeclareMathOperator{\N}{\mathbb{N}}
\DeclareMathOperator{\R}{\mathbb{R}}
\DeclareMathOperator{\uhr}{\upharpoonright} 
\DeclareMathOperator\coz{coz}
\DeclareMathOperator{\lip}{Lip}

\renewcommand{\emptyset}{\varnothing}
\numberwithin{equation}{section}

\begin{document}

\author{Valentin Gutev}
  \address{Department of Mathematics, Faculty of Science, University of
     Malta, Msida MSD 2080, Malta}
  \email{valentin.gutev@um.edu.mt}

\subjclass[2010]{26A16, 41A30, 54C20, 54C30, 54D20, 54E35, 54E40}

 \keywords{Lipschitz function, pointwise Lipschitz function, locally
   Lipschitz function, Lipschitz in the small function, extension,
   pointwise and uniform approximation}

\title{Lipschitz Extensions and Approximations}

\begin{abstract}
  The classical Hahn-Banach theorem is based on a successive
  point-by-point procedure of extending bounded linear functionals. In
  the setting of a general metric domain, the conditions are less
  restrictive and the extension is only required to be Lipschitz with
  the same Lipschitz constant. In this case, the successive procedure
  can be replaced by a much simpler one which was done by McShane
  and Whitney in the 1930s. Using virtually the same construction,
  Czipszer and Geh\'er showed a similar extension property for
  pointwise Lipschitz functions. In the present paper, we relate this
  construction to another classical result obtained previously by
  Hausdorff and dealing with pointwise Lipschitz approximations of
  semi-continuous functions. Moreover, we furnish complementary
  extension-approximation results for locally Lipschitz functions
  which fit naturally in this framework.
\end{abstract}

\date{\today}
\maketitle

\section{Introduction}

Let $(X,d)$ and $(Y,\rho)$ be metric spaces. A map $f : X\to Y$ is
called \emph{Lipschitz} if there exists $K\geq 0$ such that
\begin{equation}
  \label{eq:Loc-Lipschitz-v1:1}
  \rho(f(p),f(q))\leq K d(p,q),\quad \text{for all $p,q\in X$.}
\end{equation}
In this case, to emphasise on the constant $K$, we also say that $f$
is \emph{$K$-Lipschitz}.  The least $K\geq 0$ for which $f$ is
$K$-Lipschitz is called the \emph{Lipschitz constant} of $f$, and
denoted by $\lip(f)$.  The \emph{Lipschitz constant} of a map
$f:X\to Y$ at a non-isolated point $p\in X$ is defined by
\begin{equation}
  \label{eq:Lipschitz-ext:29}
  \lip(f,p)=\varlimsup_{x\to p}\frac{\rho(f(x),f(p))}{d(x,p)}=
  \inf_{t>0}\left[\sup_{0<d(x,p)<t}
    \frac{\rho(f(x),f(p))}{d(x,p)}\right]. 
\end{equation}
For an isolated point $p\in X$, we simply set $\lip(f,p)=0$. A map
$f:X\to Y$ is called \emph{pointwise Lipschitz} if $\lip(f,p)<+\infty$,
for every $p\in X$. For $\delta>0$, let $\mathbf{O}(p,\delta)$ be the
\emph{open $\delta$-ball centred at $p$}, i.e.
\begin{equation}
  \label{eq:Loc-Lipschitz-v18:2}
    \mathbf{O}(p,\delta)=\{x\in X:d(x,p)<\delta\}.  
\end{equation}
It follows from \eqref{eq:Lipschitz-ext:29} that $f:X\to Y$ is
pointwise Lipschitz if and only if for each $p\in X$ there
exists $L_p\geq 0$ and $\delta_p>0$ such that
\begin{equation}
  \label{eq:Lipschitz-ext:30}
  \rho(f(x),f(p))\leq L_p\, d(x,p),\quad \text{whenever $x\in
    \mathbf{O}(p,\delta_p)$.} 
\end{equation}
Finally, let us recall that a map $f:X\to Y$ is \emph{locally
  Lipschitz} if each point $p\in X$ is contained in an open set
$U\subset X$ such that $f\uhr U$ is Lipschitz.\medskip

There are simple examples of locally Lipschitz functions which are
neither Lipschitz nor uniformly continuous. On the other hand, each
locally Lipschitz map on a compact metric space is Lipschitz. In fact,
it was shown by Scanlon \cite[Theorem 2.1]{MR253053} that $f:X\to Y$
is locally Lipschitz iff its restriction on each compact subset of $X$
is Lipschitz.  According to \eqref{eq:Lipschitz-ext:30}, each
pointwise Lipschitz map is continuous and each locally Lipschitz map
is pointwise Lipschitz.  The converse is not true, see e.g.\
\cite[Example 2.7]{MR2564873}. Locally Lipschitz and pointwise
Lipschitz functions are naturally related to differentiability.  In
1919, Rademacher \cite{MR1511935} showed that each locally Lipschitz
function ${f:\Omega\to \R}$, defined on an open set
$\Omega\subset \R^n$, is differentiable almost everywhere (in the
sense of the Lebesgue measure). In 1923, Stepanov (spelled also
Stepanoff) \cite{MR1512177,zbMATH02591494} extended Rademacher's
result to pointwise Lipschitz functions. In the literature, Stepanov's
theorem is traditionally proved by applying Rademacher's theorem to
Lipschitz extensions of Lipschitz functions; a straightforward simple
proof based on Lipschitz approximations was given by J. Mal\'y
\cite{MR1687460}.\medskip

Regarding Lipschitz extensions, the following theorem was obtained by
McShane \cite[Theorem 1]{MR1562984}, see also Whitney \cite[the
footnote on p.\ 63]{MR1501735}, and is commonly called the
\emph{McShane-Whitney extension theorem}.

\begin{theorem}
  \label{theorem-Loc-Lipschitz-v1:1}
  Let $(X,d)$ be a metric space, $A\subset X$ and $\varphi:A\to \R$ be
  a $K$-Lipschitz function. Then there exists a $K$-Lipschitz function
  $\Phi:X\to \R$ which is an extension of $\varphi$, i.e.\ with
  $\Phi\uhr A=\varphi$.
\end{theorem}

The extension in Theorem \ref{theorem-Loc-Lipschitz-v1:1} was defined
by an explicit formula involving the function $\varphi$, the Lipschitz
constant $K$ and the metric $d$, see
\eqref{eq:Loc-Lipschitz-v23:2}. It is interesting to compare this
theorem with two other classical results. Namely, the construction in
Theorem \ref{theorem-Loc-Lipschitz-v1:1} is nearly the same as the one
used for the proof of the classical \emph{Hahn-Banach theorem}
\cite{zbMATH02570122,zbMATH02582128}, see also \cite{zbMATH02570123,
  58.0420.01,MR880204} and a previous work of Helly
\cite{zbMATH02625726}. This relationship was explicitly stated by
several authors, for instance in Czipszer and Geh\'er \cite{MR71493},
where they rediscovered Theorem \ref{theorem-Loc-Lipschitz-v1:1}.
This construction is also virtually the same as the one used by
Hausdorff \cite{hausdorff:19} to deal with pointwise limits of
semi-continuous functions, see \eqref{eq:Loc-Lipschitz-v23:1} and
\eqref{eq:Loc-Lipschitz-v23:3}. In fact, in his proof of a more
general result, see \cite[page 293]{hausdorff:19}, Hausdorff credited
the construction to Moritz Pasch and showed that it always gives rise
to a Lipschitz function. In contrast to the Hahn-Banach theorem, the
relationship with the Pasch-Hausdorff construction remained somehow
unnoticed. \medskip

The construction in Theorem \ref{theorem-Loc-Lipschitz-v1:1} was
further applied to pointwise Lipschitz functions by Czipszer and Geh\'er
\cite[Theorem II]{MR71493}. Namely, they proved the following
extension result.

\begin{theorem}
  \label{theorem-Loc-Lipschitz-v1:2}
  If $(X,d)$ is a metric space and $A\subset X$ is a closed subset,
  then each pointwise Lipschitz function $\varphi:A\to \R$ can be
  extended to a pointwise Lipschitz function $\Phi:X\to \R$.
\end{theorem}

In some sources, Theorem \ref{theorem-Loc-Lipschitz-v1:2} was wrongly
attributed as an extension result for locally Lipschitz functions, see
e.g.\ \cite[Theorem 4.1.7]{zbMATH07045619}. Furthermore, in other
sources, Theorem \ref{theorem-Loc-Lipschitz-v1:2} was even used as an
extension result for locally Lipschitz functions. For instance, it was
applied in \cite[Theorem 3.1]{MR253053} to show that a subset
$A\subset X$ of a metric space $(X,d)$ is closed if and only if each
(bounded) locally Lipschitz function $\varphi:A\to \R$ can be extended
to a locally Lipschitz function $\Phi:X\to \R$.  However, locally
Lipschitz extensions were not discussed by Czipszer and Geh\'er in
\cite{MR71493}. In fact, the author is not aware of any explicit
formula that will result in locally Lipschitz extensions. Such
extensions were implicitly obtained by Luukkainen and V\"ais\"al\"a in
\cite[Theorem 5.12]{MR515647} using a different technique. Namely,
Luukkainen and V\"ais\"al\"a showed that if $A$ is a closed subset of
a metric space $(X,d)$, then every locally Lipschitz map
$\varphi : A\to M$, where $M$ is an $n$-dimensional locally Lipschitz
manifold, has a locally Lipschitz extension to a neighbourhood of $A$
in $X$. Their proof was based on paracompactness of metrizable spaces,
hence relying heavily on the axiom of choice. \medskip

We are now ready to state also the main purpose of this paper. Namely,
in this paper we aim to fill in this gap and present simple and
self-contained proofs of several extension and approximation results
for Lipschitz-like functions.  In the next section, we briefly discuss
the Pasch-Hausdorff construction and its refinement in the setting of
Theorem \ref{theorem-Loc-Lipschitz-v1:1}. In Section
\ref{sec:pointw-lipsch-extens}, we also illustrate how this
construction can be applied to simplify the proof of Theorem
\ref{theorem-Loc-Lipschitz-v1:2}. In Section
\ref{sec:locally-lipsch-exten}, we show that Theorem
\ref{theorem-Loc-Lipschitz-v1:2} is still valid if ``pointwise
Lipschitz'' is replaced by ``locally Lipschitz'', see Theorem
\ref{theorem-Loc-Lipschitz-v5:1}. Our proof of this result is
implicitly based on countable paracompactness of metrizable
spaces. Briefly, using a construction of M. Mather and its refinement
for countable covers in \cite{MR1476756}, we give a direct simple
proof that any countable open cover of a metric space admits an
index-subordinated locally finite partition of unity consisting of
locally Lipschitz functions, see
Proposition~\ref{proposition-Loc-Lipschitz-v7:3}. This is a simplified
version of a more general result obtained by Z.~Frol\'{\i}k
\cite{MR814046}.  The last section of this paper contains several
applications of the approach developed in the previous sections. For
instance, we refine Haudorff's result that any bounded semi-continuous
function on a metric space is a pointwise limit of some monotone
sequence of Lipschitz functions. Namely, as commented in Remark
\ref{remark-Loc-Lipschitz-v16:3}, this is not true for unbounded
functions. Regarding the role of boundedness, we show that each
semi-continuous function defined on a metric space is a pointwise
limit of some monotone sequence of locally Lipschitz functions
(Theorem \ref{theorem-Loc-Lipschitz-v10:1}). As for semi-continuity,
we generalise the well-known property that each continuous function on
a metric space is a uniform limit of locally Lipschitz functions
(Theorem \ref{theorem-Loc-Lipschitz-v10:4}). Finally, we also give a
simple proof that each uniformly continuous function on a metric space
is a uniform limit of Lipschitz in the small functions.

\section{A Construction of Lipschitz Functions}

For a metric space $(X,d)$, a function $\varphi:X\to \R$ is
\emph{lower} (\emph{upper}) \emph{semi-continuous} if for every
$p\in X$,
\[
  \varliminf_{x\to p}\varphi(x)\geq \varphi(p)\quad
  \left(\text{respectively,}\ \varlimsup_{x\to p}\varphi(x)\leq
    \varphi(p)\right).
\]
Here,
$\varliminf_{x\to p}\varphi(x)=\sup_{\delta>0}\inf
\varphi(\mathbf{O}(p,\delta))$, see
\eqref{eq:Loc-Lipschitz-v18:2}. \medskip

Semi-continuity was introduced by Baire in his 1899 thesis
\cite{zbMATH02668286}, where he showed that each semi-continuous
function $\varphi:\R^m\to \R$ is a pointwise limit of continuous
functions, see also \cite{MR1504373,MR1400223}.  Subsequently, he
generalised this result in \cite{MR1504475} by showing that a function
$\varphi:\R^m\to \R$ is upper semi-continuous if and only if it is the
pointwise limit of a decreasing sequence of continuous functions.  In
1915, Tietze extended Baire's result to arbitrary metric
spaces. Namely, in \cite[Theorem 2]{tietze:15}, he showed that for a
metric space $(X,d)$, each upper semi-continuous bounded function
$\varphi:X\to \R$ is the pointwise limit of some decreasing sequence
of continuous functions $f_n:X\to \R$, $n\in\N$. Assuming that
$\varphi:X\to [\lambda,\mu]$ for some $\mu\geq\lambda>0$, Tietze
defined the required functions by the following explicit formula
\[
  f_n(p)=\sup_{x\in X}\frac{\varphi(x)}{[1+d(x,p)]^n},\quad p\in X.
\]
In 1919, Hausdorff \cite{hausdorff:19} gave a simple proof of this
result using another direct construction, which he credited to Moritz
Pasch. In fact, Hausdorff obtained the following stronger result as it
is evident from his proof.

\begin{theorem}[\cite{hausdorff:19}]
  \label{theorem-Loc-Lipschitz-v9:1}
  Let $(X,d)$ be a metric space and $\varphi:X\to \R$ be a function
  which is bounded from below. For each $\kappa>0$, define a function
  $f_\kappa:X\to \R$ by
  \begin{equation}
    \label{eq:Loc-Lipschitz-v23:1}
    f_\kappa(p)=\inf_{x\in X}[\varphi(x)+\kappa d(x,p)],\quad
    p\in X.
  \end{equation}
  Then $f_\kappa$ is $\kappa$-Lipschitz. If moreover $\varphi$ is
  lo\-wer semi-continuous, then $\{f_n:n\in\N\}$ is an increasing
  sequence which is pointwise convergent to $\varphi$.
\end{theorem}

\begin{proof}
  We present a brief sketch of the original proof.  For $p,q,x\in X$
  and $\kappa>0$, by the triangular inequality,
  $\kappa d(x,p)\leq \kappa d(x,q)+ \kappa d(q,p)$. Hence, by adding
  $\varphi(x)$ to both sides and taking infimum, we get that
  $f_\kappa(p)\leq f_\kappa(q)+\kappa d(q,p)$. Since $d(p,q)=d(q,p)$,
  this is equivalent to $|f_\kappa(p)-f_\kappa(q)|\leq \kappa d(p,q)$.
  It is also evident that
  $f_\kappa\leq f_{\kappa+\varepsilon}\leq \varphi$, whenever
  $\varepsilon\geq 0$. Finally, for each $n\in\N$, take $x_n\in X$
  with $\varphi(x_n)+n d(x_n,p)< f_n(p)+\frac1n$. Then
  $p=\lim_{n\to \infty}x_n$ because
  $d(x_n,p)<\frac1n\left[f_n(p)- \varphi(x_n)+\frac1n\right]\leq
  \frac1n\left[\varphi(p)- \inf_{x\in X}\varphi(x)+\frac1n\right]$. If
  $\varphi$ is lower semi-continuous, this implies that
  \[
    \varphi(p)\leq \varliminf_{x\to p}\varphi(x)\leq\varliminf_{n\to
      \infty}\varphi(x_n) \leq \lim_{n\to \infty}
    f_n(p)\leq \varphi(p).\qedhere
    \]
\end{proof}

Let us remark that in the original proof, see \cite[page
293]{hausdorff:19}, Hausdorff showed that $f_\kappa$ is
$(\kappa+1)$-Lipschitz, but the proof above is virtually the same.
Here is also a simple observation relating the case of upper
semi-continuous functions to that of Theorem
\ref{theorem-Loc-Lipschitz-v9:1}.

\begin{proposition}
  \label{proposition-Loc-Lipschitz:1}
  Let $(X,d)$ be a metric space and $\psi:X\to \R$ be a function which
  is bounded from above. For each $\kappa>0$, define a function
  $f_\kappa^*:X\to \R$ by
  \begin{equation}
    \label{eq:Loc-Lipschitz-v23:3}
    f_\kappa^*(p)=\sup_{x\in X}[\psi(x)-\kappa d(x,p)],\quad p\in X.  
  \end{equation}
  Then $f^*_\kappa=-f_k$, where $f_k$ is defined as in
  \eqref{eq:Loc-Lipschitz-v23:1} with respect to the additive inverse
  function $\varphi=-\psi$.
\end{proposition}

The Pasch-Hausdorff construction in \eqref{eq:Loc-Lipschitz-v23:1} and
its alternative interpretation in \eqref{eq:Loc-Lipschitz-v23:3} can
be applied to partial bounded functions preserving essentially the
same proof. Furthermore, they can be also applied to partial Lipschitz
functions to extend them by preserving the same Lipschitz constant.

\begin{theorem}
  \label{theorem-Loc-Lipschitz-v15:1}
  Let $(X,d)$ be a metric space, $A\subset X$ and $\varphi:A\to \R$ be
  a $\lambda$-Lipschitz function for some $\lambda\geq 0$. Define functions
  $\Phi_-:X\to (-\infty,+\infty]$ and $\Phi_+:X\to [-\infty,+\infty)$ by
  \begin{equation}
    \label{eq:Loc-Lipschitz-v23:2}
    \begin{cases}
      \Phi_-(p)=\sup_{a\in A}\left[\varphi(a)-\lambda d(a,p)\right] &\text{and}\\
     \Phi_+(p)=\inf_{a\in A}\left[\varphi(a)+\lambda d(a,p)\right],
      & p\in X.
    \end{cases}
  \end{equation}
  Then $\Phi_-\leq\Phi_+$ and\/ $\Phi_-,\Phi_+:X\to \R$ are
  $\lambda$-Lipschitz extensions of $\varphi$.
\end{theorem}

\begin{proof}
  Let $a,b\in A$ and $p\in X$. Since $\varphi$ is $\lambda$-Lipschitz
  \[
    \varphi(a)-\varphi(b)\leq \lambda d(a,b)\leq \lambda
    d(a,p)+ \lambda d(p,b).
  \]
  Accordingly, $\Phi_-(p)\leq \Phi_+(p)$ because
  $\varphi(a)-\lambda d(a,p)\leq \varphi(b)+\lambda d(b,p)$. By taking
  $p\in A$, this implies that
  $\varphi(p)\leq \Phi_-(p)\leq \Phi_+(p)\leq \varphi(p)$, so
  $\Phi_-\uhr A=\varphi= \Phi_+\uhr A$.  In fact, the function
  $\Phi_+$ corresponds to the function $f_\lambda$ in
  \eqref{eq:Loc-Lipschitz-v23:1}, the only difference is that now the
  infimum is taken on the points of $A$. Hence, the same argument as
  in Theorem \ref{theorem-Loc-Lipschitz-v9:1} shows that it is
  $\lambda$-Lipschitz.  Accordingly, $\Phi_-$ is also
  $\lambda$-Lipschitz, see Proposition
  \ref{proposition-Loc-Lipschitz:1}. 
\end{proof}

\begin{remark}
  \label{remark-Loc-Lipschitz-v16:1}
  The Lipschitz extension $\Phi_-$ in Theorem
  \ref{theorem-Loc-Lipschitz-v15:1} represents McShane and Whitney's
  approach to show Theorem \ref{theorem-Loc-Lipschitz-v1:1}. The other
  extension $\Phi_+$ was used by Czipszer and Geh\'er for their
  alternative proof of this theorem. The advantage of using both
  functions $\Phi_-$ and $\Phi_+$, as done in Theorem
  \ref{theorem-Loc-Lipschitz-v15:1}, lies in the simplification of the
  proof. Furthermore, as pointed out implicitly in \cite{MR71493},
  $\Phi_-\leq f\leq \Phi_+$ for every $\lambda$-Lipschitz extension
  $f:X\to \R$ of $\varphi:A\to \R$.  Indeed, such a function $f$ is a
  $\lambda$-Lipschitz extension of itself and by Theorem
  \ref{theorem-Loc-Lipschitz-v15:1},
  \[
    \Phi_-(p)=\sup_{a\in
    A}[\varphi(a)-\lambda d(a,p)]\leq \sup_{x\in
    X}[f(x)-\lambda d(x,p)]=f(p),\quad p\in X.
\]
Similarly, $f\leq \Phi_+$.  Thus, $\Phi_-$ is the smallest possible
$\lambda$-Lipschitz extension of $\varphi$, while $\Phi_+$ --- the
largest one.\qed
\end{remark}

\begin{remark}
  \label{remark-Loc-Lipschitz-v16:2}
  For a metric space $(X,d)$, a bounded function $\varphi:X\to \R$ and
  $\kappa>0$, let $f_\kappa:X\to \R$ be defined as in
  \eqref{eq:Loc-Lipschitz-v23:1}. Then $f_\kappa$ is the greatest
  $\kappa$-Lipschitz function with $f_\kappa\leq \varphi$. Indeed, for
  a $\kappa$-Lipschitz function $f:X\to \R$ with $f\leq \varphi$, by
  Theorem \ref{theorem-Loc-Lipschitz-v15:1} (applied with $A=X$ and
  $\varphi=f$),
  \[
    f(p)=\inf_{x\in X}\left[f(x)+\kappa d(x,p)\right]
  \leq \inf_{x\in X}\left[\varphi(x)+\kappa
    d(x,p)\right]=f_\kappa(p),\quad p\in X.
  \]
  Based on this property, the functions $f_\kappa$, $\kappa>0$, are
  often called the \emph{Pasch-Hausdorff envelope} of $\varphi$. One
  can easily see that $f=\lim_{n\to \infty}f_n$ is the largest lower
  semi-continuous function $f:X\to \R$ with $f\leq \varphi$. In fact,
  $f$ is the well-known \emph{lower Baire function} associated to
  $\varphi$ because $f(p)=\varliminf_{x\to p}\varphi(x)$, $p\in
  X$. \qed
\end{remark}

\begin{remark}
  \label{remark-Loc-Lipschitz-v16:3}
  The restriction in Theorem \ref{theorem-Loc-Lipschitz-v9:1} that
  $\varphi$ is bounded from below is necessary. Indeed, let $X=\R$ be
  the real line endowed with an equivalent bounded metric, for
  instance $d(x,y)=\frac{|x-y|}{1+|x-y|}$, $x,y\in\R$, and
  $\varphi:X\to \R$ be the identity. Then for any
  $\kappa$-Lipschitz function $f:X\to \R$, as in Remark
  \ref{remark-Loc-Lipschitz-v16:2}, it follows that
  \[
    \inf_{x\in X}\left[x+\kappa d(x,p)\right]=-\infty <f(p)=
    \inf_{x\in X}\left[f(x)+\kappa d(x,p)\right], \quad p\in X.
  \]
  Accordingly, the inequality $f\leq \varphi$ is impossible.\qed
\end{remark}

\section{Pointwise-Lipschitz Extensions}
\label{sec:pointw-lipsch-extens}

Let $(X,d)$ and $(Y,\rho)$ be metric spaces. Following the Lipschitz
condition \eqref{eq:Loc-Lipschitz-v1:1}, we shall say that a map
$f:X\to Y$ is \emph{globally
  Lipschitz} at a point $p\in X$ if there exists $K_p\geq 0$ such that 
\begin{equation}
  \label{eq:Loc-Lipschitz-v4:1}
  \rho(f(x),f(p))\leq K_p\, d(x,p),\quad \text{for every $x\in X$.}
\end{equation}
In this case, $K_p$ will be called a \emph{global Lipschitz constant
  at $p\in X$}. Here is an example of maps which are globally
Lipschitz at each point of their domain. 

\begin{proposition}
  \label{proposition-Loc-Lipschitz-v3:2}
  If $(X,d)$ and $(Y,\rho)$ are metric spaces, then each bounded
  pointwise Lipschitz map $f:X\to Y$ is globally Lipschitz at each
  point of $X$. 
\end{proposition}

\begin{proof}
  Let $M\geq 0$ be such that $\rho(f(x),f(z))\leq M$, for all
  $x,z\in X$. Also, for a point $p\in X$, let $L_p\geq 0$ and
  $\delta_p>0$ be as in \eqref{eq:Lipschitz-ext:30}. Then
  $K_p= \max\left\{L_p,\frac M{\delta_p}\right\}$ is as in
  \eqref{eq:Loc-Lipschitz-v4:1}.
\end{proof}

Proposition \ref{proposition-Loc-Lipschitz-v3:2} allows to apply the
construction in Theorem \ref{theorem-Loc-Lipschitz-v15:1}, but now
using some fixed global Lipschitz constants of a bounded pointwise
Lipschitz function. Namely, let $(X,d)$ be a metric space and
$A\subset X$. Whenever $\varphi:A\to \R$ is a bounded pointwise
Lipschitz function and $\lambda_a$, $a\in A$, are fixed global
Lipschitz constants of $\varphi$, following
\eqref{eq:Loc-Lipschitz-v23:2}, we may associate the pair of extended
functions $\Phi_-:X\to (-\infty,+\infty]$ and
$\Phi_+:X\to [-\infty,+\infty)$ defined by
  \begin{equation}
    \label{eq:Loc-Lipschitz-v4:2}
    \begin{cases}
      \Phi_-(p)=\sup_{a\in A}\left[\varphi(a)-\lambda_a d(a,p)\right] &\text{and}\\
      \Phi_+(p)=\inf_{a\in A}\left[\varphi(a)+\lambda_a
        d(a,p)\right], &\text{for every $p\in X$.}
    \end{cases}
  \end{equation}

  In this general setting, just as in Theorem
  \ref{theorem-Loc-Lipschitz-v15:1}, we have that 
\begin{equation}
  \label{eq:Loc-Lipschitz-v18:1}
  \Phi_-\leq \Phi_+\quad\text{and}\quad \Phi_-\uhr
  A=\varphi=\Phi_+\uhr A. 
\end{equation}
Indeed, for $a,b\in A$, let
$\lambda_{ab}=\min\{\lambda_a,\lambda_b\}$. Then for every $p\in X$,
it follows from \eqref{eq:Loc-Lipschitz-v4:1} that
$\varphi(a)-\varphi(b)\leq \lambda_{ab} d(a,b)\leq \lambda_a d(a,p) +
\lambda_b d(b,p)$.  In other words,
$\varphi(a)-\lambda_a d(a,p)\leq \varphi(b)+\lambda_b d(b,p)$ which is
clearly equivalent to $\Phi_-(p)\leq \Phi_+(p)$. Similarly, taking
$p\in A$, it follows from \eqref{eq:Loc-Lipschitz-v4:2} that
$\varphi(p)\leq \Phi_-(p)\leq \Phi_+(p)\leq \varphi(p)$, therefore
$\Phi_-\uhr A=\varphi= \Phi_+\uhr A$. \medskip

The approach of using both functions $\Phi_-$ and $\Phi_+$ is also
successful in simplifying the original proof of Theorem
\ref{theorem-Loc-Lipschitz-v1:2}, see \cite[Theorem II]{MR71493}. The
simplification is essentially in the following special case of this
theorem.

\begin{theorem}
  \label{theorem-Loc-Lipschitz-v17:1}
  If $(X,d)$ is a metric space, $A\subset X$ is a closed set and
  $M>0$, then each pointwise Lipschitz function $\varphi:A\to (-M,M)$
  can be extended to a pointwise Lipschitz function
  $\Phi:X\to (-M,M)$.
\end{theorem}

\begin{proof}
  By Proposition~\ref{proposition-Loc-Lipschitz-v3:2}, a pointwise
  Lipschitz function $\varphi:A\to (-M,M)$ is globally Lipschitz at
  each point of $A$. So, we may associate the extended functions
  $\Phi_-:X\to (-\infty,+\infty]$ and $\Phi_+:X\to [-\infty,+\infty)$
  defined as in \eqref{eq:Loc-Lipschitz-v4:2} with respect to some
  global Lipschitz constants $\lambda_a$, $a\in A$, of $\varphi$.
  Then by \eqref{eq:Loc-Lipschitz-v18:1}, $\Phi_-$ and $\Phi_+$ are
  extensions of $\varphi$ with $\Phi_-\leq \Phi_+$ and, in particular,
  usual functions.  We will show that they are pointwise Lipschitz as
  well. Since $-\varphi:A\to (-M,M)$ is also pointwise Lipschitz, just
  as in Proposition \ref{proposition-Loc-Lipschitz:1}, this is reduced
  to showing that one of these functions, say $\Phi_+$, is pointwise
  Lipschitz.  So, take a point $p\in X$.\medskip

  If $p\in A$ and $x\in X$, then by \eqref{eq:Loc-Lipschitz-v18:1},
  $\Phi_-(x)\leq \Phi_+(x)$ and $\Phi_+(p)=\varphi(p)$. Hence, it
  follows from \eqref{eq:Loc-Lipschitz-v4:2} that
  \[
    \Phi_+(p)-\lambda_p d(p,x)\leq \Phi_-(x)\leq \Phi_+(x)\leq
    \Phi_+(p)+ \lambda_p d(p,x).
  \]
  Thus, $|\Phi_+(x)-\Phi_+(p)|\leq \lambda_p d(p,x)$.\medskip

  Otherwise, if $p\notin A$, we will show that $\Phi_+$ is Lipschitz
  in the open $\delta$-ball $\mathbf{O}(p,\delta)$, see
  \eqref{eq:Loc-Lipschitz-v18:2}, where $\delta=\frac{d(p,A)}2>0$ is
  the half distance between $p$ and the closed set $A$. To this end,
  fix a point $c\in A$ with $d(c,p)\leq 3\delta$. Then by
  \eqref{eq:Loc-Lipschitz-v4:2},
  \begin{equation}
    \label{eq:Loc-Lipschitz-v19:1}
    \Phi_+(x)\leq \varphi(c)+\lambda_{c}
    d(c,x) \leq M+ 4\delta \lambda_{c}=M_c,\quad \text{for all $x\in
    \mathbf{O}(p,\delta)$.}
  \end{equation}    
  Take now $x,y\in \mathbf{O}(p,\delta)$ and any $\varepsilon>0$. By
  \eqref{eq:Loc-Lipschitz-v4:2},
  $\Phi_+(x)+\varepsilon\geq \varphi(a)+ \lambda_a d(a,x)$ for some
  $a\in A$.  Therefore, by \eqref{eq:Loc-Lipschitz-v19:1},
  $\lambda_a\leq \frac{\Phi_+(x)-\varphi(a)+\varepsilon}{d(a,x)}\leq
  \frac{M_c+M+\varepsilon}\delta$ because $d(a,x)\geq\delta$.  For
  this $a\in A$, we also have that
  $\Phi_+(y)\leq \varphi(a)+\lambda_a d(a,y)$. Accordingly,
  \begin{align*}
    \Phi_{+}(y)-\Phi_+(x) &\leq \lambda_a d(y,a)-\lambda_a d(x,a)+\varepsilon\\
                        &\leq \lambda_a d(y,x)+\varepsilon\leq
                          \frac{M_c+M+\varepsilon}\delta d(y,x)+ 
                          \varepsilon.
  \end{align*}
  Thus, $\Phi_+(y)-\Phi_+(x) \leq \frac{M_c+M}\delta d(y,x)$ and, by
  symmetry, $\Phi_+$ is Lipschitz in $\mathbf{O}(p,\delta)$.\smallskip

  We may now define $\Phi=\frac{\Phi_-+\Phi_+}{2+d(\cdot, A)}$, where
  $d(\cdot,A)$ is the distance function. Then $\Phi$ remains a
  pointwise Lipschitz extension of $\varphi$ because
  $\frac1{2+d(\cdot,A)}$ is a bounded Lipschitz function whose
  restriction on $A$ is the constant $\frac12$. Moreover, $\Phi$ takes
  values in $(-M,M)$. Indeed, this is evident for the points of $A$
  because $\Phi\uhr A=\varphi$. For the rest of the points of $X$, it
  suffices to see that $|\Phi_+(p)+\Phi_-(p)|\leq 2M$ because
  $d(p,A)>0$, whenever $p\notin A$. The latter property follows easily
  from \eqref{eq:Loc-Lipschitz-v4:2} because for $a\in A$, we have
  that   $\Phi_+(p) +\varphi(a) -\lambda_a d(a,p)\leq 2\varphi(a)\leq 2M$.
  Hence, taking supremum on $A$, we get that $\Phi_+(p)+\Phi_-(p)\leq
  2M$. Similarly, we also have that $\Phi_-(p)+\Phi_+(p)\geq
  -2M$.
\end{proof}

We conclude this section with several remarks regarding the proper
place of Theorem \ref{theorem-Loc-Lipschitz-v17:1}.

\begin{remark}
  \label{remark-Loc-Lipschitz-v23:1}
  Theorem \ref{theorem-Loc-Lipschitz-v1:2} follows easily from Theorem
  \ref{theorem-Loc-Lipschitz-v17:1}. For instance, as in
  \cite{MR71493}, to a pointwise Lipschitz function $\varphi:A\to \R$
  one can associate the function
  \[
    \tilde{\varphi}=\arctan\circ \varphi:A\xrightarrow[]{\varphi} \R
    \xrightarrow[]{\arctan} \left(-\frac\pi2,\frac\pi2\right).
  \]
  Then $\tilde{\varphi}$ is also pointwise Lipschitz because
  $\arctan:\R\to\left(-\frac\pi2,\frac\pi2\right)$ is itself
  Lipschitz.  Hence, by Theorem \ref{theorem-Loc-Lipschitz-v17:1},
  $\tilde{\varphi}$ has a pointwise Lipschitz extension
  ${\tilde{\Phi}:X\to
    \left(-\frac\pi2,\frac\pi2\right)}$. Finally, since
  $\tan:\left(-\frac\pi2,\frac\pi2\right)\to
  \R$ is locally Lipschitz, we may define the required extension of
  $\varphi$ by $\Phi=\tan\circ \tilde{\Phi}:X\to \R$.\qed
\end{remark}

\begin{remark}
  \label{remark-Loc-Lipschitz-v18:1}
  One can easily see that in Theorems \ref{theorem-Loc-Lipschitz-v1:2}
  and \ref{theorem-Loc-Lipschitz-v17:1}, the requirement $A$ to be
  closed is essential. For instance, the function
  $\frac{t}{|t|}:[-1,0)\cup(0,1]\to \{-1,1\}$ is even locally
  Lipschitz, but cannot be extended continuously on $[-1,1]$.\qed
\end{remark}

\begin{remark}
  \label{remark-Loc-Lipschitz-v18:2}
  According to the proof of Theorem \ref{theorem-Loc-Lipschitz-v17:1},
  the extensions $\Phi_-$ and $\Phi_+$, defined as in
  \eqref{eq:Loc-Lipschitz-v4:2}, are locally Lipschitz at the points
  of $X\setminus A$. Our proof of this fact is essentially the same as
  the original proof in \cite[Theorem II]{MR71493}. In contrast, the
  original proof that $\Phi_+$ is pointwise Lipschitz at the points of
  $A$ is too complicated. The same proof of this fact was also given
  in \cite[Theorem 4.1.7]{zbMATH07045619}. In both sources, for points
  $p\in A$ and $x\in X$, it was used the definition of $\Phi_+$ in
  \eqref{eq:Loc-Lipschitz-v4:2} to get at once that
  $\Phi_+(x)-\Phi_+(p)\leq \lambda_p d(x,p)$. However, the other
  inequality that $\Phi_+(x)-\Phi_+(p)\geq -\lambda_p d(x,p)$ was
  obtained by taking $a\in A\setminus\{p\}$ and manipulating the
  inequalities $\varphi(a)-\varphi(p)\geq -\lambda_a d(a,p)$ and
  $\varphi(a)-\varphi(p)\geq -\lambda_p d(a,p)$; in fact, by
  multiplying the first one by $\frac{d(a,x)}{d(a,x)+d(x,p)}$ and the
  second one by $\frac{d(x,p)}{d(a,x)+d(x,p)}$. The other benefit of
  using both functions $\Phi_-$ and $\Phi_+$ is the immediate
  construction of the extension $\Phi:X\to (-M,M)$. Namely, in both
  \cite{MR71493} and \cite{zbMATH07045619}, this extension is obtained
  by refining the construction of $\Phi_+$ to get a pointwise
  Lipschitz extension $\Psi_+:X\to (-M,M]$ of $\varphi$. Next,
  applying the same construction to $-\varphi$, to get a similar
  pointwise Lipschitz extension $\Psi_-:X\to [-M,M)$ of
  $-\varphi$. Then the required extension was defined by
  $\Psi=\frac{\Psi_+-\Psi_-}2$. \qed
\end{remark}

\section{Locally Lipschitz Extensions}
\label{sec:locally-lipsch-exten}

In this section, we will show the following extension theorem.

\begin{theorem}
  \label{theorem-Loc-Lipschitz-v5:1}
  Let $(X,d)$ be a metric space and $A\subset X$ be a closed set. Then
  each locally Lipschitz function $\varphi:A\to \R$ can be extended to a
  locally Lipschitz function $\Phi:X\to \R$. 
\end{theorem}

As mentioned in the Introduction, a special case of Theorem
\ref{theorem-Loc-Lipschitz-v5:1} follows from a more general result,
see \cite[Theorem 5.12]{MR515647}, which is based on paracompactness
of metrizable spaces. Here, we use a similar approach but based only
on countable paracompactness. Recall that a space $X$ is
\emph{countably paracompact} if each countable open cover of $X$ has a
locally finite open refinement.  The fact that each metrizable
space is countably paracompact is a simple consequence of Dowker's
result \cite[Theorem 2]{dowker:51}, see also \cite[Theorem
1.1]{MR776636}. Here, we will use this implicitly based on the
following considerations.\medskip

The \emph{cozero set}, or the \emph{set-theoretic support}, of a
function $\xi : X\to \R$ is the set
$\coz(\xi) = \{x\in X : \xi(x)\neq 0\}$.  A collection
$\xi_n:X\to [0,1]$, $n\in \N$, of continuous functions on a space $X$
is a \emph{partition of unity} if $\sum_{n=1}^\infty\xi_n(x)=1$, for
each $x\in X$.  A partition of unity $\{\xi_n:n\in \N\}$ is
\emph{locally finite} if $\{\coz(\xi_n) : n\in \N\}$ is a locally
finite cover of $X$; and it is \emph{index-subordinated} to a cover
$\{O_n:n\in \N\}$ of $X$ if $\coz(\xi_n)\subset O_n$, for every
$n\in \N$. It was shown in \cite[Lemma]{MR814046} that each countable
open cover of a metric space admits a locally finite
index-subordinated partition of unity consisting of Lipschitz
functions. Here, we give a direct proof of the following special case
of this result. 

\begin{proposition}
  \label{proposition-Loc-Lipschitz-v7:3}
  Each open cover $\{O_n:n\in\N\}$ of a metric space $(X,d)$ admits a
  locally finite index-subordinated partition of unity consisting of
  locally Lipschitz functions.
\end{proposition}

\begin{proof}
  If $X=O_n$ for some $n\in \N$, the property is trivial. Suppose that
  $O_n\neq X$ for every $n\in\N$.  Following the proof of
  \cite[Lemma]{MR814046}, for each $n\in\N$, consider the
  $1$-Lipschitz function $\eta_n:X\to [0,2^{-n}]$ defined by
  $\eta_n(x)=\min\{d(x,X\setminus O_n), 2^{-n}\}$, $x\in X$. Then
  $\coz(\eta_n)=O_n$, $n\in\N$, and the function
  $\eta=\sum_{n=1}^\infty\frac{\eta_n}{2^n}:X\to [0,1]$ remains
  $1$-Lipschitz. Moreover, $\eta$ is positive-valued because
  $\{\coz(\eta_n):n\in\N\}$ is a cover of $X$. We may now use a
  construction M. Mather, see \cite[Lemma]{MR0281155} and \cite[Lemma
  5.1.8]{engelking:89}, in fact its further refinement for countable
  covers in \cite{MR1476756}. Namely, for each $n\in\N$, define a
  function
  $\gamma_n(x)=\max\left\{\eta_n(x)-\frac12 \eta(x),0\right\}$,
  $x\in X$. Evidently, $0\leq\gamma_n\leq \eta_n$ and $\gamma_n$ is
  Lipschitz because so are $\eta_n$ and $\eta$. If $p\in X$, then
  $\eta(p)>2^{-k}$ for some $k\in \N$, hence $p$ is contained in an
  open set $V\subset X$ with $\eta(x)>2^{-k}$ for every $x\in
  X$. Accordingly, $\gamma_n(x)=0$ for every $x\in V$ and $n> k$, so
  $\{\coz(\gamma_n):n\in \N\}$ is locally finite. It is also a cover
  of $X$. Indeed, for $p\in X$, we have that
  $\sup_{n\in\N}\eta_n(p)\geq 2^{-n_0}>\eta_n(p)$ for some $n_0\in \N$
  and every $n> n_0$. Therefore, $\sup_{n\in\N}\eta_n(p)=\eta_m(p)$
  for some $m\leq n_0$. This implies that
  $\eta(p)\leq \sum_{n=1}^\infty
  \frac{\eta_m(p)}{2^n}=\eta_m(p)$. Accordingly,
  $\eta_m(p)>\frac12 \eta(p)$ and $\gamma_m(p)>0$. We may now define
  the required partition of unity by
  $\xi_k=\frac{\gamma_k}{\sum_{n=1}^\infty \gamma_n}$, $k\in\N$. It
  consists of locally Lipschitz functions because $\sum_{n=1}^\infty
  \gamma_n$ is a bounded function and locally it is a finite sum of
  Lipschitz functions. 
\end{proof}

Proposition \ref{proposition-Loc-Lipschitz-v7:3} is in good accord
with the following property of bounded locally Lipschitz functions. 

\begin{proposition}
  \label{proposition-Loc-Lipschitz-v5:2}
  If $(X,d)$ and $(Y,\rho)$ are metric spaces and $f:X\to Y$ is a
  bounded locally Lipschitz map, then $X$ has an open increasing cover
  $\{U_n:n\in\N\}$ such that $f\uhr U_n$ is Lipschitz, for each
  $n\in \N$.
\end{proposition}

\begin{proof}
  Let $M\geq 0$ be such that $\rho(f(x),f(z))\leq M$, for every
  $x,z\in X$.  Also, for $p\in X$, let $\delta_p>0$ and $L_p\geq 0$ be
  such that $\rho(f(x),f(z))\leq L_p\, d(x,z)$, for every
  $x,z\in \mathbf{O}(p,\delta_p)$. As in Proposition
  \ref{proposition-Loc-Lipschitz-v3:2}, setting
  $K_p=\max\left\{L_p,\frac M{\delta_p}\right\}$, it follows that
  \begin{equation}
    \label{eq:Loc-Lipschitz-v6:1}
    \rho(f(x),f(z))\leq K_p\, d(x,z),\quad \text{for every $x\in X$
      and $z\in \mathbf{O}(p,\delta_p)$.}
  \end{equation}
  Finally, set
  $U_n=\bigcup\{\mathbf{O}(p,\delta_p): K_p\leq n\}$,
  $n\in\N$. Evidently, $\{U_n:n\in\N\}$ is an increasing open cover of
  $X$. Moreover, if $x,z\in U_n$, then $x\in \mathbf{O}(p,\delta_p)$
  and $z\in \mathbf{O}(q,\delta_q)$ for some $p,q\in X$ with
  $\max\{K_p,K_q\}\leq n$. According to
  \eqref{eq:Loc-Lipschitz-v6:1}, this implies that
  $\rho(f(x),f(z))\leq n d(x,z)$.
\end{proof}
  
Based on this, we have the following extension result.

\begin{proposition}
  \label{proposition-Loc-Lipschitz-v6:1}
  Let $(X,d)$ be a metric space and $A\subset X$ be a closed
  subset. Then each bounded locally Lipschitz function
  $\varphi:A\to \R$ can be extended to a locally Lipschitz function
  $\Phi:X\to \R$.
\end{proposition}

\begin{proof}
  By Proposition \ref{proposition-Loc-Lipschitz-v5:2}, $A$ has an open
  cover $\{U_n:n\in\N\}$ such that each restriction
  $\varphi_n=\varphi\uhr U_n$, $n\in\N$, is Lipschitz. Since $A$ is
  closed, this cover can be extended to an open cover $\{O_n:n\in\N\}$
  of $X$, i.e.\ with the property that $O_n\cap A=U_n$, $n\in\N$. Then
  by Proposition \ref{proposition-Loc-Lipschitz-v7:3},
  $\{O_n:n\in\N\}$ has an index-subordinated locally finite partition
  of unity $\{\xi_n:n\in\N\}$ consisting of locally Lipschitz
  functions.  Moreover, by Theorem \ref{theorem-Loc-Lipschitz-v1:1},
  each $\varphi_n:U_n\to \R$ can be extended to a Lipschitz function
  $\Phi_n:X\to \R$. We can now define a function $\Phi:X\to \R$ by
  $\Phi(x)=\sum_{n=1}^\infty \xi_n(x)\cdot \Phi_n(x)$, $x\in X$. Then
  $\Phi$ remains locally Lipschitz because each $p\in X$ has a
  neighbourhood $V\subset X$ such that
  $\{n\in \N: \coz(\xi_n)\cap V\neq \emptyset\}$ is finite, therefore
  $\Phi\uhr V$ is a finite sum of locally Lipschitz
  functions. Moreover, $\Phi$ is an extension of $\varphi$. Indeed,
  take a point $a\in A$ and set $\sigma_a=\{n\in\N:
  \xi_n(a)>0\}$. Then $\Phi_n(a)=\varphi_n(a)=\varphi(a)$ for every
  $n\in\sigma_a$, and therefore
  \[
    \Phi(a)=\sum_{n\in\sigma_a} \xi_n(a)\cdot \Phi_n(a)
    =\varphi(a)\cdot \sum_{n\in\sigma_a} \xi_n(a)= \varphi(a)\cdot
    1=\varphi(a).\qedhere
  \]
\end{proof}
  
Finally, we also have the
following refinement of Proposition
\ref{proposition-Loc-Lipschitz-v6:1}.

\begin{corollary}
  \label{corollary-Loc-Lipschitz-v9:1}
  If $(X,d)$ is a metric space, $A\subset X$ is a closed set and
  $M>0$, then each locally Lipschitz function $\varphi:A\to (-M,M)$
  can be extended to a locally Lipschitz function $\Phi:X\to (-M,M)$.
\end{corollary}

\begin{proof}
  If $\varphi:A\to (-M,M)$ is locally Lipschitz, then by Proposition
  \ref{proposition-Loc-Lipschitz-v6:1}, it can be extended to a
  locally Lipschitz function $\Psi:X\to \R$. Let
  $U=\Psi^{-1}\big((-M,M)\big)$, which is an open set containing
  $A$. Hence, there is a locally Lipschitz function
  $\eta:X\to [0,1]$ with $\eta^{-1}(0)=X\setminus U$ and
  $\eta^{-1}(1)=A$. For instance, $\eta\equiv 1$ is the constant
  $1$ if $U=X$; otherwise we can take
  $\eta=\frac{d(\cdot,X\setminus U)}{d(\cdot,A)+d(\cdot,X\setminus
    U)}$ because the distance functions $d(\cdot,X\setminus U)$ and
  $d(\cdot,A)$ are Lipschitz. Then the product function
  $\Phi=\eta\cdot\Psi:X\to (-M,M)$ is as required.
\end{proof}

\begin{proof}[Proof of Theorem \ref{theorem-Loc-Lipschitz-v5:1}]
  We can proceed as in Remark \ref{remark-Loc-Lipschitz-v23:1}.
  Briefly, for a locally Lipschitz function $\varphi:A\to \R$, the
  function
  $ \tilde{\varphi}=\arctan\circ \varphi:A\to
  \left(-\frac\pi2,\frac\pi2\right)$ is also locally Lipschitz.
  Hence, by Corollary \ref{corollary-Loc-Lipschitz-v9:1}, it can be
  extended to a locally Lipschitz function
  ${\tilde{\Phi}:X\to \left(-\frac\pi2,\frac\pi2\right)}$. Finally,
  since $\tan:\left(-\frac\pi2,\frac\pi2\right)$ is locally Lipschitz,
  the function 
  $\Phi=\tan\circ \tilde{\Phi}:X\to \R$ is as required.
\end{proof}

We conclude this section with the following two remarks. 

\begin{remark}
  \label{remark-Loc-Lipschitz-v20:1}
  The construction in Proposition \ref{proposition-Loc-Lipschitz-v6:1}
  is standard and can be applied to give a direct proof of Theorem
  \ref{theorem-Loc-Lipschitz-v5:1}. To this end, let us recall that an
  arbitrary collection $\xi_\alpha:X\to [0,1]$,
  $\alpha\in \mathscr{A}$, of continuous functions on a space $X$ is a
  \emph{partition of unity} if
  $\sum_{\alpha\in \mathscr{A}}\xi_\alpha(x)=1$, for each $x\in X$.
  Here, ``$\sum_{\alpha\in \mathscr{A}}\xi_\alpha(x)=1$'' means that
  only countably many functions $\xi_\alpha$'s do not vanish at $x$,
  and the series composed by them is convergent to 1. Now, suppose
  that $(X,d)$, $A\subset X$ and $\varphi:A\to \R$ are as in Theorem
  \ref{theorem-Loc-Lipschitz-v5:1}. Then by definition, each point
  $a\in A$ has a neighbourhood $U_a\subset A$ such that the
  restriction $\varphi_a=\varphi\uhr U_a$ is Lipschitz. Since $A$ is
  closed, one can take an open cover $\{O_a:a\in A\}$ of $X$ such that
  $O_a\cap A=U_a$, $a\in A$. Also, using Theorem
  \ref{theorem-Loc-Lipschitz-v1:1}, one can extend each
  $\varphi_a:U_a\to \R$ to a Lipschitz function $\Phi_a:X\to
  \R$. Thus, the proof is reduced to the existence of a locally finite
  partition of unity $\{\xi_a:a\in A\}$ which is index-subordinated to
  $\{O_a:a\in A\}$ and consists of locally Lipschitz functions.  The
  construction of such a partition of unity is not a simple fact and
  will require the use of A. H. Stone's theorem \cite[Corollary
  1]{stone:48} that each metrizable space is paracompact, also the
  Lefschetz lemma \cite{MR0007093} that each point-finite open cover
  of a normal space has a closed shrinking. However, both these
  results rely heavily on the axiom of choice.  That is, this approach
  will make the proof unnecessarily complicated.\qed
\end{remark}

\begin{remark}
  \label{remark-Loc-Lipschitz-v20:2}
  It was shown by Dowker \cite[Theorem 4]{dowker:51} that a space $X$
  is countably paracompact and normal if and only if for each pair of
  functions $\varphi,\psi:X\to \R$ such that $\varphi$ is upper
  semi-continuous, $\psi$ is lower semi-continuous and $\varphi<\psi$,
  there exists a continuous function $f:X\to \R$ with
  $\varphi <f<\psi$. For a metric space $(X,d)$, this result was
  partially refined in \cite[Theorem 5.4]{MR515647} by showing that
  $f$ can be constructed to be locally Lipschitz. The argument
  suggested in \cite{MR515647} is to follow an alternative proof of
  this theorem given in \cite[4.3, p. 171]{dugundji:66} which is based
  on paracompactness and partitions of unity, but now to take a
  locally Lipschitz partition of unity. In this regard, let us
  explicitly remark that the locally Lipschitz partitions of unity
  considered in \cite{MR515647} are based on paracompactness of
  metrizable spaces, see Remark
  \ref{remark-Loc-Lipschitz-v20:1}. Using Proposition
  \ref{proposition-Loc-Lipschitz-v7:3}, one can get a simple direct
  proof of this fact. Briefly, as in \cite{dugundji:66}, take an open
  cover $\{O_r: r\in\Q\}$ of $X$ such that $\varphi(x)<r<\psi(x)$, for
  every $x\in O_r$ and $r\in\Q$. Since the rational numbers $\Q$ are
  countable, by Proposition \ref{proposition-Loc-Lipschitz-v7:3},
  $\{O_r: r\in\Q\}$ has an index-subordinated locally finite partition
  of unity $\{\xi_r:r\in\Q\}$ consisting of locally Lipschitz
  functions. The required locally Lipschitz function $f:X\to \R$ is
  now defined by $f(x)=\sum_{r\in\Q}\xi_r(x)\cdot r$, $x\in X$, see
  the proof of Proposition \ref{proposition-Loc-Lipschitz-v6:1}. \qed
\end{remark}

\section{Pointwise and Uniform Approximations}

As shown in Remark \ref{remark-Loc-Lipschitz-v16:3}, the approximation
property in Theorem \ref{theorem-Loc-Lipschitz-v9:1} is not valid for
unbounded functions. For such functions, using the method developed in
the previous section, we have the following refined property.

\begin{theorem}
  \label{theorem-Loc-Lipschitz-v10:1}
  If $(X,d)$ is a metric space, then each lower semi-continuous
  function $\varphi:X\to \R$ is a pointwise limit of some increasing
  sequence of locally Lipschitz functions.
\end{theorem}

\begin{proof}
  Let $\varphi:X\to \R$ be lower semi-continuous. Then for each
  $k\in\N$, the set $O_k=\{x\in X:\varphi(x)>-k\}$ is open. Moreover,
  $X=\bigcup_{k=1}^\infty O_k$ and by Proposition
  \ref{proposition-Loc-Lipschitz-v7:3}, there exists a locally finite
  partition of unity $\{\xi_k:k\in\N\}$ of locally Lipschitz functions
  such that $\coz(\xi_k)\subset O_k$, $k\in\N$. On the other hand, we
  may define a sequence $\varphi_k:X\to \R$, $k\in\N$, of (bounded
  from below) lower semi-continuous functions by
  \begin{equation}
    \label{eq:Loc-Lipschitz-v10:1}
    \varphi_k(x)=
    \begin{cases}
      \varphi(x) &\text{if $x\in O_k$,}\\
      -k &\text{otherwise.}
    \end{cases}
  \end{equation}
  Then by Theorem \ref{theorem-Loc-Lipschitz-v9:1}, each $\varphi_k$
  is the pointwise limit of some increasing sequence $f^k_n:X\to \R$,
  $n\in\N$, of Lipschitz functions. We may now define the required
  sequence of functions $f_n:X\to \R$, $n\in\N$, by
  $f_n(x)=\sum_{k=1}^\infty \xi_k(x)\cdot f^k_n(x)$, $x\in X$.
  Indeed, as in the proof of Proposition
  \ref{proposition-Loc-Lipschitz-v6:1}, each $f_n$ is locally
  Lipschitz. It is also evident that $f_n\leq f_{n+1}$ because
  $f^k_n\leq f^k_{n+1}$, for every $k\in \N$. Finally, take a point
  $p\in X$ and set $\sigma_p=\{k\in\N:\xi_k(p)>0\}$. Then
  $f_n(p)=\sum_{k\in\sigma_p}\xi_k(p)\cdot f_n^k(p)$ and by
  \eqref{eq:Loc-Lipschitz-v10:1},
  $\lim_{n\to \infty}f^k_n(p)=\varphi(p)$ for every $k\in
  \sigma_p$. Since $\sigma_p$ is a finite set, this implies that
  $\lim_{n\to\infty}f_n(p)=\sum_{k\in\sigma_p} \xi_k(p)\cdot
  \varphi(p)= \varphi(p)\cdot 1=\varphi(p)$.
\end{proof}

In case the function $\varphi:X\to \R$ in Theorem
\ref{theorem-Loc-Lipschitz-v10:1} is continuous (i.e.\ both lower and
upper semi-continuous), it can be uniformly approximated by locally
Lipschitz functions. The following result is an alternative version of
\cite[Theorem 5.17]{MR515647} which doesn't make use of
paracompactness of metrizable spaces.

\begin{theorem}
  \label{theorem-Loc-Lipschitz-v10:4}
  Let $(X,d)$ be a metric space, $\varphi:X\to \R$ and
  $\varepsilon:X\to (0,+\infty)$ be continuous functions, and
  $A\subset X$ be a closed set. Then each locally Lipschitz function
  $g:A\to \R$ with $|g(a)-\varphi(a)|<\varepsilon(a)$, $a\in A$, can
  be extended to a locally Lipschitz function $f:X\to \R$ such that
  $|f(x)-\varphi(x)|<\varepsilon(x)$, for every $x\in X$.
\end{theorem}

Regarding the proper place of Theorem
\ref{theorem-Loc-Lipschitz-v10:4}, let us remark that the range in
\cite[Theorem 5.17]{MR515647} is an $n$-dimensional locally Lipschitz
manifold, while the locally Lipschitz extension is controlled by an
$\varepsilon$-homotopy, hence representing an
$\varepsilon$-approximation as well. For the real line, one can simply
take a linear homotopy $H:X\times [0,1]\to \R$ between $f$ and
$\varphi$, namely $H(x,t)=(1-t)\cdot f(x)+ t\cdot \varphi(x)$. Then
$|H(x,t)-\varphi(x)|<\varepsilon(x)$, for every $x\in X$ and
$t\in[0,1]$. In fact, Theorem \ref{theorem-Loc-Lipschitz-v10:4} is
also a natural generalisation of Theorem
\ref{theorem-Loc-Lipschitz-v5:1} because by Tietze's extension theorem
\cite{tietze:15}, each continuous function $g:A\to \R$, defined on a
closed subset $A$ of a metrizable space $X$, can be extended to a
continuous function $\varphi:X\to \R$. \medskip

A word should be also said about the role of the function
$\varepsilon:X\to (0,+\infty)$. For a metric space $(Y,\rho)$, a
natural generalisation of the uniform topology on the function space
$C(X,Y)$ of all continuous maps from $X$ to $Y$ is the \emph{fine
  topology}. It is generated by all sets
$\{g\in C(X,Y): \rho(g(x),f(x))<\varepsilon(x),\ x\in X\}$, where
$f\in C(X,Y)$ and $\varepsilon$ runs on $C(X,(0,+\infty))$.  This
topology was first introduced by Whitney \cite{whitney:36}, and still
sometimes carries his name. In contrast to the uniform topology, the
fine topology does not depend on the metric on $Y$ provided $X$ is
countably paracompact and normal \cite{dimaio-hola-holy-mccoy:98}; for
a paracompact $X$, this was previously obtained by Whitehead
\cite{MR124917} and Krikorian \cite{krikorian:69}. Thus, for a metric
space $(X,d)$, Theorem \ref{theorem-Loc-Lipschitz-v10:4} shows that
the locally Lipschitz members of $C(X,\R)$ are dense in $C(X,\R)$ with
respect not only to the uniform topology, but also with respect to the
fine topology.\medskip

The proof of Theorem \ref{theorem-Loc-Lipschitz-v10:4} is based on its
special case for the uniform topology. The fact that each continuous
function on a metric space can be uniformly approximated by locally
Lipschitz functions was obtained by several authors, see for instance
Miculescu \cite[Theorem 2]{MR1825525} and Garrido and Jaramillo
\cite[Corollary 2.8]{MR2037989}. Below, we present a simple and
self-contained proof of this property.

\begin{proposition}
  \label{proposition-Loc-Lipschitz-v21:1}
  Whenever $(X,d)$ is a metric space, each continuous function
  $\varphi:X\to \R$ is a uniform limit of locally Lipschitz functions.
\end{proposition}

\begin{proof}
  If $\varphi:X\to \R$ is continuous and $\varepsilon>0$, then
  $\left\{\varphi^{-1}\left(\mathbf{O}(r,\varepsilon)\right):
    r\in\Q\right\}$ is a countable open cover of $X$. Hence, by
  Proposition \ref{proposition-Loc-Lipschitz-v7:3}, it has a locally
  finite index-subordinated partition of unity $\{\xi_r:r\in\Q\}$
  consisting of locally Lipschitz functions. As in Remark
  \ref{remark-Loc-Lipschitz-v20:2}, we may define a locally Lipschitz
  function $f:X\to\R$ by $f(x)=\sum_{r\in\Q}\xi_r(x)\cdot r$,
  $x\in X$. If $x\in X$, then the set
  $\sigma_x=\{r\in \Q: \xi_r(x)\neq 0\}$ is finite. Since
  $\sum_{r\in\sigma_x}\xi_r(x)=1$, this implies that
  \begin{align*}
    |f(x)-\varphi(x)|&=
    \left|\sum_{r\in\sigma_x}\xi_r(x)\cdot
     r-\sum_{r\in\sigma_x}\xi_r(x)\cdot \varphi(x)\right|\\
    &\leq \sum_{r\in\sigma_x}\xi_r(x)\cdot|r-\varphi(x)|<
    \sum_{r\in\sigma_x}\xi_r(x)\cdot \varepsilon=\varepsilon.\qedhere 
  \end{align*}
\end{proof}

The property in Proposition \ref{proposition-Loc-Lipschitz-v21:1} can
be easily extended from a constant $\varepsilon>0$ to a (lower semi-)
continuous function $\varepsilon:X\to (0,+\infty)$.

\begin{corollary}
  \label{corollary-Loc-Lipschitz-v21:1}
  Let $(X,d)$ be a metric space and $\varphi:X\to \R$ be a continuous
  function. Then for every continuous function
  $\varepsilon:X\to (0,+\infty)$, there exists a locally Lipschitz
  function $f:X\to \R$ with $|f(x)-\varphi(x)|<\varepsilon(x)$, for
  every $x\in X$.
\end{corollary}

\begin{proof}
  Consider the open cover
  $O_n=\left\{x\in X: \varepsilon(x)>\frac 1n\right\}$, $n\in\N$, of
  $X$ and take a locally finite index-subordinated partition of unity
  $\{\xi_n:n\in\N\}$ consisting of locally Lipschitz functions. Also,
  by Proposition \ref{proposition-Loc-Lipschitz-v21:1}, for each
  $n\in\N$ take a locally Lipschitz function $f_n:X\to \R$ such that
  $|f_n(x)-\varphi(x)|<\frac1n$, $x\in X$. Then the function
  $f(x)=\sum_{n=1}^\infty \xi_n(x)\cdot f_n(x)$, $x\in X$, is as
  required because $x\in O_n$ implies that
  $|f_n(x)-\varphi(x)|<\varepsilon(x)$.
\end{proof}

\begin{proof}[Proof of Theorem \ref{theorem-Loc-Lipschitz-v10:4}]
  We can use the idea in the proof of Corollary
  \ref{corollary-Loc-Lipschitz-v9:1}. Namely, by Theorem
  \ref{theorem-Loc-Lipschitz-v5:1}, $g$ can be extended to a locally
  Lipschitz function $\Psi:X\to \R$.  Then
  $U=\{x\in X: |\Psi(x)-\varphi(x)|<\varepsilon(x)\}$ is an open set
  containing $A$. Hence, there is a locally Lipschitz function
  $\eta:X\to [0,1]$ such that $X\setminus U=\eta^{-1}(0)$ and
  $A=\eta^{-1}(1)$, see the proof of Corollary
  \ref{corollary-Loc-Lipschitz-v9:1}. Moreover, by Corollary
  \ref{corollary-Loc-Lipschitz-v21:1}, there is a locally Lipschitz
  function $\Phi:X\to \R$ such that
  $|\Phi(x)-\varphi(x)|<\varepsilon(x)$, for every $x\in X$. Finally,
  define $f:X\to \R$ by
  $f(x)=\eta(x)\cdot \Psi(x)+ (1-\eta(x))\cdot \Phi(x)$, $x\in
  X$. This $f$ is as required.
\end{proof}

We conclude this paper with a result about uniform Lipschitz-like
approximations of uniformly continuous functions. To this end, let us
recall two other classes of ``locally'' Lipschitz functions. For
metric spaces $(X,d)$ and $(Y,\rho)$, a map $f:X\to Y$ is called
\emph{uniformly locally Lipschitz} if there exists $\delta>0$ such
that $f\uhr \mathbf{O}(x,\delta)$ is Lipschitz, for each $x\in X$. A
map $f:X\to Y$ is called \emph{Lipschitz in the small} \cite{MR538094}
if there exists $\delta>0$ and $K>0$ such that each restriction
$f\uhr \mathbf{O}(x,\delta)$ is $K$-Lipschitz, for every $x\in
X$. There are standard examples showing that these Lipschitz-like maps
are different. For instance, $\frac1t:(0,+\infty)\to \R$ is locally
Lipschitz but not uniformly locally Lipschitz, while the function
$t^2:\R\to \R$ is uniformly locally Lipschitz but not Lipschitz in
the small. Also, one can easily construct Lipschitz in the small
functions which are not Lipschitz. For instance, such a function is
the one defined in Remark \ref{remark-Loc-Lipschitz-v16:3}. This is
not surprising because in the realm of bounded maps $f:X\to Y$,
Lipschitz is the same as Lipschitz in the small, see
\cite[2.15]{MR538094} and Proposition
\ref{proposition-Loc-Lipschitz-v3:2}.\medskip

The Lipschitz in the small maps are naturally related to
uniform approximations of uniformly continuous maps.

\begin{proposition}
  \label{proposition-Loc-Lipschitz-v21:2}
  Let $(X,d)$ and $(Y,\rho)$ be metric spaces and $\varphi:X\to Y$
  be a map which is a uniform limit of Lipschitz in the small
  maps. Then $\varphi$ is uniformly continuous.
\end{proposition}

\begin{proof}
  Let $\varepsilon>0$ and $f:X\to Y$ be a map which is Lipschitz in
  the small such that $\rho(f(x),\varphi(x))<\frac\varepsilon3$ for
  every $x\in X$. By definition, there are $\delta>0$ and $K>0$ such
  that $f\uhr \mathbf{O}(x,\delta)$ is $K$-Lipschitz, for every
  $x\in X$.  Then taking
  $\delta_0=\min\left\{\frac{\varepsilon}{3K},\delta\right\}$, it
  follows that $\rho(\varphi(x),\varphi(z))<\varepsilon$ for every
  $x,z\in X$ with $d(x,z)<\delta_0$.
\end{proof}

In case $Y=\R$ is the real line, the converse is also true. The fact
that each uniformly continuous function $\varphi:X\to \R$ is a uniform
limit of Lipschitz in the small functions was observed by several
authors. For a compact metric space $X$, the result was obtained by
Georganopoulos \cite{MR214994}. For a bounded uniformly continuous
function $\varphi:X\to \R$, the result can be found in \cite[Theorem
6.8]{MR1800917}, see also \cite[Theorem 1]{MR1973966}. In fact, the
proof in \cite{MR1800917} is very simple and consists of showing that
the sequence of Lipschitz functions defined as in
\eqref{eq:Loc-Lipschitz-v23:1} for $\kappa=n\in\N$ is uniformly
convergent to $\varphi$. For an arbitrary uniformly continuous
function, the result was obtained by Garrido and Jaramillo
\cite[Theorem 1]{MR2376153}. Subsequently, a simple proof of this
theorem was given by Beer and Garrido in \cite[Theorem
6.1]{MR3334948}, which was a refined version of the argument in
\cite{MR1800917}. Below we simplify further this proof illustrating
its natural relationship with Theorem
\ref{theorem-Loc-Lipschitz-v9:1}.

\begin{theorem}
  \label{theorem-Loc-Lipschitz-v10:3}
  If $(X,d)$ is a metric space, then each uniformly continuous
  function $\varphi:X\to \R$ is a uniform limit of functions which are
  Lipschitz in the small.
\end{theorem}

\begin{proof}
  Suppose that $\varphi:X\to \R$ is uniformly continuous and
  $\varepsilon>0$. Also, let $\delta>0$ be such that
  $|\varphi(x)-\varphi(y)|<\varepsilon$, for every $x,y\in X$ with
  $d(x,y)<2\delta$. Next, following \cite{MR3334948}, take $k\in\N$
  such that $k\delta > \varepsilon$, and define $f : X\to \R$ by
  \begin{equation}
    \label{eq:Loc-Lipschitz-v10:3}
    f(x)=\inf_{y\in \mathbf{O}(x,2\delta)}[\varphi(y)+ k d(y,x)],\quad
    x\in X.
  \end{equation}
  Then $\varphi(x)\geq f(x)\geq \varphi(x)-\varepsilon$ because
  $\varphi(y)>\varphi(x)-\varepsilon$ for every
  $y\in \mathbf{O}(x,2\delta)$, so $|f(x)-\varphi(x)|< \varepsilon$.
  Whenever $y\in \mathbf{O}(x,2\delta)$ with $d(y,x)\geq
  \delta$, since $k\delta>\varepsilon$, it follows that
  $ f(x)\leq \varphi(x)<\varphi(y)+\varepsilon<\varphi(y)+k\delta\leq
  \varphi(y)+k d(y,x)$.  Therefore, we actually have that
  \begin{equation}
    \label{eq:Loc-Lipschitz-v10:4}
    f(x)=\inf_{y\in
      \mathbf{O}\left(x,\delta\right)}\left[\varphi(y)+k
      d(y,x)\right],\quad x\in X. 
  \end{equation}
  This implies that $f$ is $k$-Lipschitz on each open
  $\frac\delta2$-ball. Indeed, take $p,q\in X$ with
  $d(p,q)<\delta$. Then
  $\mathbf{O}(p,\delta)\cup \mathbf{O}(q,\delta)\subset
  \mathbf{O}(p,2\delta)\cap \mathbf{O}(q,2\delta)=X_{pq}$. Since
  $k d(x,p)\leq k d(x,q)+ k d(q,p)$, $x\in X_{pq}$, as in the proof of
  Theorem \ref{theorem-Loc-Lipschitz-v9:1} but now using
  \eqref{eq:Loc-Lipschitz-v10:3} and \eqref{eq:Loc-Lipschitz-v10:4},
  we get that $f(p)\leq f(q)+k d(p,q)$. Evidently, this is
  equivalent to $|f(p)-f(q)|\leq k d(p,q)$.
\end{proof}

\providecommand{\bysame}{\leavevmode\hbox to3em{\hrulefill}\thinspace}


\end{document}